\documentclass[11pt,reqno,a4paper]{amsart}

\usepackage{color}
\usepackage[colorlinks=true, allcolors=blue]{hyperref}
\usepackage{amsmath, amssymb, amsthm}
\usepackage{mathrsfs}
\usepackage{mathtools}
\usepackage[noabbrev,capitalize,nameinlink]{cleveref}
\crefname{equation}{}{}
\usepackage{fullpage}
\usepackage[noadjust]{cite}
\usepackage{graphics}
\usepackage{pifont}
\usepackage{tikz}
\usepackage{bbm}
\usepackage[T1]{fontenc}

\usetikzlibrary{arrows.meta}

\usepackage{environ}
\usepackage{framed}
\usepackage{url}
\catcode`~=11    \def\UrlSpecials{\do\~{\kern -.15em\lower .7ex\hbox{~}\kern .04em}}  \catcode`~=13 %
\usepackage[linesnumbered,ruled,vlined]{algorithm2e}
\usepackage[noend]{algpseudocode}
\usepackage[labelfont=bf]{caption}
\usepackage{cite}
\usepackage{framed}
\usepackage[framemethod=tikz]{mdframed}
\tikzstyle{vertex}=[circle,draw=black,fill=black,inner sep=0,minimum size=0.2cm,text=white,font=\footnotesize]
\usepackage{appendix}
\usepackage{graphicx}
\usepackage[textsize=tiny]{todonotes}
\usepackage{tcolorbox}
\usepackage{enumerate}
\allowdisplaybreaks[1]
\usepackage{enumerate}

\usepackage[utf8]{inputenc} 
\usepackage[T1]{fontenc}

\crefname{algocf}{Algorithm}{Algorithms}

\crefname{equation}{}{} 
\AtBeginEnvironment{appendices}{\crefalias{section}{appendix}} 

\usepackage[color,final]{showkeys} 

\colorlet{refkey}{orange!20}
\colorlet{labelkey}{blue!30}

\crefname{algocf}{Algorithm}{Algorithms}

\numberwithin{equation}{section}
\newtheorem{theorem}{Theorem}[section]

\newtheorem{lemma}[theorem]{Lemma}

\crefname{claim}{Claim}{Claims}

\newtheorem*{question*}{Question}

\theoremstyle{definition}
\newtheorem{definition}[theorem]{Definition}

\newtheorem*{definition*}{Definition}

\theoremstyle{remark}

\newcommand{\bs}{\boldsymbol}
\newcommand{\mb}{\mathbb}
\newcommand{\mbf}{\mathbf}

\newcommand{\mc}{\mathcal}

\newcommand{\on}{\operatorname}

\allowdisplaybreaks

\title{Large Deviations in Random Latin Squares}

\author[Kwan]{Matthew Kwan}

\address{Institute of Science and Technology Austria (IST Austria), 3400 Klosterneuburg, Austria}

\email{matthew.kwan@ist.ac.at}

\author[A2]{Ashwin Sah}
\author[A3]{Mehtaab Sawhney}
\address{Department of Mathematics, Massachusetts Institute of Technology, Cambridge, MA}
\email{\{asah,msawhney\}@mit.edu}

\thanks{Kwan was supported by NSF grant DMS-1953990. Sah and Sawhney were supported by NSF Graduate Research Fellowship Program DGE-1745302.}

\begin{document}

\begin{abstract}
In this note, we study large deviations of the number $\mathbf{N}$
of \emph{intercalates} ($2\times2$ combinatorial subsquares which are themselves Latin squares) in a random
$n\times n$ Latin square. In particular, for constant $\delta>0$ we
prove that $\exp(-O(n^{2}\log n))\le \Pr(\mathbf{N}\le(1-\delta)n^{2}/4)\le\exp(-\Omega(n^{2}))$
and $\exp(-O(n^{4/3}(\log n)))\le \Pr(\mathbf{N}\ge(1+\delta)n^{2}/4)\le\exp(-\Omega(n^{4/3}(\log n)^{2/3}))$.
As a consequence, we deduce that a typical order-$n$ Latin square has $(1+o(1))n^{2}/4$
intercalates, matching a lower bound due to Kwan and Sudakov and resolving an old conjecture of McKay and Wanless.
\end{abstract}

\maketitle

\section{Introduction} \label{sec:introduction}

A \emph{Latin square} (of order $n$) is an $n\times n$ array filled
with the numbers $1$ through $n$ (we call these \emph{symbols}),
such that every symbol appears exactly once in each row and column.
Latin squares are a fundamental type of combinatorial design, and
in their various guises they play an important role in many contexts
(ranging, for example, from group theory, to experimental design,
to the theory of error-correcting codes). A classical introduction
to the subject of Latin squares can be found in \cite{KD15}. More recently,
Latin squares have also played a role in the ``high-dimensional combinatorics''
program spearheaded by Linial, where they can be viewed as the first
nontrivial case of a ``high-dimensional permutation''\footnote{To see the analogy to permutation matrices, note that a Latin square
can equivalently, and more symmetrically, be viewed as an $n\times n\times n$
zero-one array such that every axis-aligned line sums to exactly 1.} (see for example \cite{LL14,LL16,LS18}).

There are still a number of surprisingly basic questions about Latin
squares that remain unanswered, especially with regard to statistical
aspects. For example, there is still a big gap between
the best known upper and lower bounds on the number of order-$n$
Latin squares (see for example \cite[Chapter~17]{vW01}), and there is no known algorithm that (provably) efficiently generates a random order-$n$ Latin square\footnote{Jacobson and Matthews~\cite{JM96} and Pittenger~\cite{Pit97} designed Markov
chains that converge to the uniform distribution, but it is not known
whether these Markov chains mix rapidly.}. Perhaps the main difficulty is that Latin squares are extremely
``rigid'' objects: in general there is very little freedom to make
local perturbations to change one Latin square into another.

Despite this difficulty, there are a number of theorems that have been
rigorously proved about random Latin squares (and a larger number
of conjectures and speculations); see for example \cite{LS18,Kwa20,KS18,CGW08,Cam92,HJ96,Cam15,MW05,LL16,van90,Wan11,MW99}. A large portion
of this work has focused on existence and enumeration of various types
of substructures. As perhaps the simplest nontrivial example, an \emph{intercalate}
in a Latin square $L$ is an order-2 Latin (combinatorial) subsquare. That is, it is
a pair of rows $i<j$ and a pair of columns $x<y$ such that $L_{i,x}=L_{j,y}$
and $L_{i,y}=L_{j,x}$ (see \cref{fig:latin}). It is a classical fact that
(for all orders except 2 and 4) there exist Latin squares with no
intercalates~\cite{KLR75,KT76,McL75}. However, in 1999 McKay and Wanless~\cite{MW99}
proved that with probability $1-\exp(-n^{2-o(1)})$
a random order-$n$ Latin square has at least one intercalate, and
that with probability $1-o(1)$ there are at least $n^{3/2-o(1)}$
intercalates. In the same paper, they conjectured that the
typical number of intercalates is $(1+o(1))n^{2}/4$.
More recently, Kwan and Sudakov~\cite{KS18} proved the lower bound in
this conjecture---that random Latin squares typically have \emph{at
least} this many intercalates (see also \cite{CGW08} for previous progress on this conjecture). In the present paper we finally resolve
McKay and Wanless' conjecture in full.

\begin{figure}
\begin{centering}
\begin{tabular}{|c|c|c|c|c|}
\hline 
4 & \textbf{1} & 5 & \textbf{3} & 2\tabularnewline
\hline 
5 & \textbf{3} & 2 & \textbf{1} & 4\tabularnewline
\hline 
2 & 4 & 1 & 5 & 3\tabularnewline
\hline 
3 & 5 & 4 & 2 & 1\tabularnewline
\hline 
1 & 2 & 3 & 4 & 5\tabularnewline
\hline 
\end{tabular}\qquad\qquad\qquad\qquad
$\vcenter{\hbox{\includegraphics{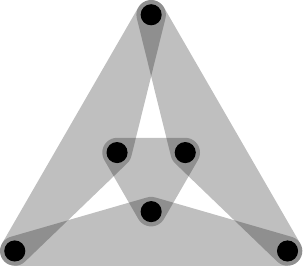}}}$






\par\end{centering}
\caption{\label{fig:latin}On the left is an example of a Latin square of order 5 with an intercalate
in bold (in the first and second rows, and the second and fourth columns).
On the right is depicted the 3-uniform hypergraph representation of
an intercalate.}
\end{figure}

\begin{theorem}\label{thm:MW}
Let $\mathbf{L}$ be a uniformly random order-$n$ Latin square.
Then, with probability $1-o(1)$, the number of intercalates
in $\mathbf{L}$ is $(1+o(1))n^{2}/4$.
\end{theorem}

It is natural to draw an analogy to small subgraph counts in random
graphs and hypergraphs. For example, in an Erd\H os--R\'enyi random
graph $\mb{G}(n,p)$, the number of triangles is typically
close to its expected value of $\binom{n}{3}p^{3}$ (as may be proved
with a routine application of Chebyshev's inequality). There is no
obvious way to compute almost any kind of expected value in random
Latin squares, but this point of view at least gives a heuristic explanation
for why one should expect \cref{thm:MW} to hold, as follows. An order-$n$ Latin
square can be equivalently viewed as a 3-partite 3-uniform hypergraph
with parts of size $n$ (corresponding to rows, columns, and symbols), satisfying the property that every pair of
vertices in different parts is included in exactly one (hyper)edge.
In this setting an intercalate is a subgraph isomorphic to a particular
4-edge hypergraph; see \cref{fig:latin}. Now, every Latin square has exactly
$n^{2}$ edges, so by symmetry, in a random Latin square each of the
$n^{3}$ possible edges is present with probability $1/n$. If we
imagine that each of these edges were present with probability $1/n$
\emph{independently}, then the expected number of intercalates would
be $2\binom{n}{2}^{3}(1/n)^{4}=(1+o(1))n^{2}/4$.

There are a huge number of questions about subgraph counts in random
graphs and hypergraphs that have natural analogues for random Latin
squares. One particularly influential direction is the study of \emph{large
deviations}. For example, what is the probability that a random graph
$\mb{G}(n,p)$ has more than twice as many triangles as expected?
What is the probability it has fewer than half as many as expected?
These types of questions have been intensely studied and are intimately related to the development of many important techniques in graph theory and probability theory; see for example the monograph of Chatterjee \cite{Cha17} and the more recent works \cite{Aug20,BB19,BGLZ17,CD20,HMS19}. Beyond \cref{thm:MW}, we are able to prove the following
near-optimal bounds on large deviation probabilities for intercalates
in random Latin squares.
\begin{theorem}\label{thm:large-deviations}
Fix a constant $\delta>0$. Let $\mathbf{N}$ be the number of intercalates
in a uniformly random order-$n$ Latin square $\mathbf{L}$. Then
\begin{enumerate}
\item [(a)]$\Pr(\mathbf{N}\le(1-\delta)n^{2}/4)\le\exp(-\Omega(n^{2}))$,
\item [(b)]$\Pr(\mathbf{N}\ge(1+\delta)n^{2}/4)\le\exp(-\Omega(n^{4/3}(\log n)^{2/3}))$.
\end{enumerate}
Moreover, these bounds are best-possible up to logarithmic factors
in the exponent:
\begin{enumerate}
\item [(c)]$\Pr(\mathbf{N}\le(1-\delta)n^{2}/4)\ge\exp(-O(n^{2}\log n))$ for $\delta\le 1$,
\item [(d)]$\Pr(\mathbf{N}\ge(1+\delta)n^{2}/4)\ge\exp(-O(n^{4/3}\log n))$.
\end{enumerate}
\end{theorem}

Note that \cref{thm:MW} is a direct corollary of \cref{thm:large-deviations}(a--b). Another direct corollary of \cref{thm:large-deviations}(a) is that $\Pr(\mathbf{N}=0)\le\exp(\Omega(-n^{2}))$,
improving McKay and Wanless' aforementioned bound of $\exp(-n^{2-o(1)})$. We remark that the lower tail bound proved by Kwan and Sudakov~\cite{KS18} was of the form $\Pr(\mathbf{N}\le(1-\delta)n^{2}/4)\le\exp(-\Omega(\sqrt n/\log n))$.

From the form of the upper and lower tail probabilities in \cref{thm:large-deviations}, one
can already begin to get an idea for why the upper bound in \cref{thm:MW} is
more difficult than the lower bound. In general, for subgraph counts
in random graphs and hypergraphs, lower tails tend to behave in a
relatively simple ``Gaussian-like'' way, while upper tails tend
to be quite different due to ``clustering'' behaviour (for example,
in some regimes the ``most likely way'' for a random graph to have
a large number of triangles is for it to contain a large clique that
has many triangles on its own). This phenomenon is often referred
to as the ``infamous upper tail'' (see \cite{JR02} for a survey).
In the setting of \cref{thm:large-deviations}, it seems that the ``most likely way'' for
a random Latin square to have a large number of intercalates is for
it to contain a configuration similar to the multiplication table
of an abelian 2-group $(\mb{Z}/2\mb{Z})^{q}$ (which may be
interpreted as a Latin square of order $2^{q}$), for suitably chosen
$q$.

We remark that as a na\"ive approach to try to prove \cref{thm:large-deviations}, we might
try to study the independent random hypergraph model mentioned earlier
(in which each edge is present with probability $1/n$ independently),
and to condition on the (hopefully not too unlikely) event that our
random hypergraph is in fact a Latin square. For example, it is possible
to study large deviations in random regular graphs with a related approach~\cite{BD19,Gun20} (although the details are highly nontrivial). However, the property of being a Latin square
is extremely restrictive, and there does not seem to be any simple
independent model that produces a Latin square with probability
greater than about $(1/\sqrt{n})^{n^{2}}$ (which is vanishingly
small compared to the large deviation probabilities in \cref{thm:large-deviations}(a--b)). Therefore,
we employ some techniques not commonly seen in large deviations theory.

The upper and lower tails in \cref{thm:large-deviations} are handled quite differently. For
the lower tail, we employ the powerful machinery of Keevash (see \cite{Kee14,Kee18,Kee18b,Kee18c})
originally developed for his celebrated proof of the \emph{existence
of designs} conjecture. Using Keevash's machinery, Kwan \cite{Kwa20} developed
a general method for comparing random Latin squares with a stochastic
graph process called the \emph{triangle removal process}. It has been
observed by Simkin \cite{Sim18} that this method is suitable for bounding
lower tail probabilities, but to prove the strong bound in \cref{thm:large-deviations}(a), we
need to refine Kwan's method (introducing an additional averaging
technique).

For the upper tail bound in \cref{thm:large-deviations}(b), instead of working directly with random Latin
squares we work with random \emph{Latin rectangles} (a Latin rectangle
is a $k\times n$ array, for some $k\le n$, filled with the symbols 1 through $n$, such
that every number appears at most once in each row and column). As
observed by McKay and Wanless, we can use estimates on the permanent
(Bregman's theorem~\cite{Bre73} and the Egorychev--Falikman theorem~\cite{Ego81,Fal81}) to compare
random Latin rectangles with random Latin squares. To study random
Latin rectangles we use the method of \emph{switchings} (in which
we study the typical effect of random perturbations to a Latin rectangle),
in connection with a general enumeration theorem of Godsil and McKay~\cite{GM90}
and the so-called \emph{deletion method} of R\"odl and Ruci\'nski (see \cite{RR95,JR04}), adapted to this highly non-independent situation.

\subsection{Further Directions}
There are a few natural questions left open by our work. Let $\mathbf N$ be the number of intercalates in a random order-$n$ Latin square.
\begin{itemize}
    \item Can we improve our understanding of the large deviation probabilities for $\mathbf N$, and sharpen the logarithmic factors\footnote{We note that in the case of triangles in random graphs, it was a longstanding open problem to find the correct logarithmic factor in the exponent of the upper tail probability. This was famously solved by Chatterjee~\cite{Cha12} and DeMarco and Kahn~\cite{DK12}.} in \cref{thm:large-deviations}? In particular, it seems that \cref{thm:large-deviations}(c) is improvable, but the difficulty lies in finding a general way to complete partial Latin squares to Latin squares without introducing too many intercalates. It seems that Keevash's machinery may not be suitable for this, but the more recent approach of ``iterative absorption'' due to Glock, K\"uhn, Lo and Osthus~\cite{GKLO} (see also \cite{DGKLMO20}) may be helpful here.
    \item It would be nice to obtain a more accurate understanding of the expected value $\mathbb E \mathbf N$, and to say more about the distribution of $\mathbf N$ (in particular, it is not even obvious how to estimate the variance of $\mathbf N$). With the ideas in this paper it is possible to find an explicit interval of length $n^{2-\Omega(1)}$ in which $\mathbf N$ typically lies, but we suspect the true behaviour is that $\mathbf N$ has an asymptotic Gaussian distribution with standard deviation $\Theta(n)$.
    \item We have studied $2\times 2$ Latin subsquares; of course it is natural to consider subsquares of higher order. McKay and Wanless~\cite{MW99} conjectured that the expected number of $3\times 3$ Latin subsquares is $1/18+o(1)$ (we would further conjecture that the distribution is asymptotically Poisson with this mean), and they suggested that Latin subsquares of higher order should typically not appear at all. They also proved that $n/2\times n/2$ subsquares are vanishingly unlikely in a random order-$n$ Latin square (this is the largest a proper Latin subsquare could possibly be). We suspect it may not be too hard to show that a typical order-$n$ Latin square does not contain a proper Latin subsquare of order greater than $n^{1/2+\varepsilon}$ (for any constant $\varepsilon>0$), and it would be interesting to go beyond this. Of course, it is also possible to study more general subgraph statistics: for any fixed partial Latin square $H$, we can ask about the number of copies of $H$ in a random Latin square.
    \item A \emph{Steiner triple system} of order $n$ is a 3-uniform hypergraph on a vertex set of size $n$, such that every pair of vertices is included in exactly one edge. These objects are natural ``non-partite'' analogues of Latin squares, and are even more difficult to study (to our knowledge, the only nontrivial results about random Steiner triple systems can be found in \cite{Bab80,Kwa20,FK20,Sim18}). In the setting of Steiner triple systems, the 4-edge hypergraph we have been calling an intercalate is usually called a \emph{Pasch configuration}. Pasch configurations represent the smallest nontrivial ``girth'' obstruction for Steiner triple systems (Erd\H os conjectured that there exist Steiner triple systems with arbitrarily high girth; see \cite{Erd76}), and they provide one of very few ways to ``switch'' between different Steiner triple systems. Simkin~\cite{Sim18} adapted some ideas of Kwan~\cite{Kwa20} to prove that a random Steiner triple system typically has at least $(1-o(1))n^2/24$ Pasch configurations (and the ideas in this paper are suitable for proving near-optimal bounds on the lower-tail probabilities), but due to the ``infamous upper tail'' it will require new ideas to prove a corresponding bound for the upper tail.
    \item We would also like to draw attention to a few other interesting open problems in the area of random Latin squares that are a bit less directly related to the results in this paper. Linial and Luria~\cite{LL16} conjectured that random Latin squares typically satisfy an expansion property closely resembling the expander mixing lemma (see \cite{KS18} for progress on this conjecture) and Cavenagh, Greenhill and Wanless~\cite{CGW08} conjectured that a fixed pair of rows in a random Latin square can be very closely approximated (in some precise sense) by a uniformly random derangement (see also \cite{Cam15} for further discussion).
\end{itemize}

\subsection{Notation}\label{subsec:not}
We use standard asymptotic notation throughout, as follows. For functions
$f=f(n)$ and $g=g(n)$, we write $f=O(g)$
to mean that there is a constant $C$ such that $|f|\le C|g|$,
$f=\Omega(g)$ to mean that there is a constant $c>0$
such that $f(n)\ge c|g(n)|$ for sufficiently large $n$, $f=\Theta(g)$ to mean
that that $f=O(g)$ and $f=\Omega(g)$, and $f=o(g)$ to mean that $f/g\to0$ as $n\to\infty$.
Also, following \cite{Kee14}, the notation $f=1\pm\varepsilon$ means
$1-\varepsilon\le f\le1+\varepsilon$.

We will use the convention that random objects (for example, random variables or random graphs) are printed in bold.

\subsection*{Acknowledgements}
We thank Zach Hunter for pointing out some important typographical errors. We also thank the referee for several remarks which helped improve the paper substantially.

\section{Approximation for random Latin squares}\label{sec:approximation}

In this section we state and prove a refined version of a theorem due to Kwan~\cite{Kwa20} (\cref{thm:approximation}), using machinery due to Keevash~\cite{Kee18c} to approximate a random Latin square with the so-called triangle removal process. This will be the main technical ingredient for the proof of \cref{thm:large-deviations}(a).

To say a bit more about our contribution: Kwan's original approximation theorem (\cite[Theorem~2.4]{Kwa20}) is not capable of proving that any events hold with probability less than $e^{-n}$, so is not sufficient for proving the extremely strong lower tail bound in \cref{thm:large-deviations}(a). Our improvement comes from an averaging/double counting technique (for the reader familiar with the proof of \cite[Theorem~2.4]{Kwa20}, instead of conditioning on an outcome of a random subset of a random Latin square, we average over many subsets). This averaging technique is closely related to the ``distance to hyperplane'' lemma in work of Rudelson and Vershynin~\cite{RV08}, which is ubiquitous in random matrix theory. 

First, we need some definitions, including the (equivalent) hypergraph formulation of a Latin square.

\begin{definition}\label{def:latin-square}
Define
\[
R=R_{n}=\{ 1,\dots,n\} ,\quad C=C_{n}=\{ n+1,\dots,2n\} ,\quad S=S_{n}=\{ 2n+1,\dots,3n\}.
\]
We call the elements of $R$, $C$ and $S$ \emph{rows, columns, }and\emph{
symbols }respectively. A \emph{partial Latin square} (of order $n$)
is a 3-partite 3-uniform hypergraph with 3-partition $R\cup C\cup S$, such that
no pair of vertices is involved in more than one edge. Let $\mathcal{L}_{m}$
be the set of partial Latin squares with $m$ edges. A \emph{Latin
square} is a partial Latin square with exactly $n^{2}$ edges (this
is the maximum possible, and implies that every pair of vertices in different parts is contained in exactly one edge). Let $\mathcal{L}$ be the set of Latin squares.
\end{definition}

\begin{definition}
The (3-partite) \emph{triangle removal process} is defined as follows. Start with
the complete 3-partite graph $K_{n,n,n}$ on the vertex set $R\cup C\cup S$.
At each step, consider the set of all triangles in the current graph,
select one uniformly at random, and remove it. Note that after $m$
steps of this process, the set of removed triangles can be interpreted
as a partial Latin square $L\in\mathcal{L}_{m}$ (unless we run out
of triangles before the $m$th step). Let $\mathbb{L}(n,m)$
be the distribution on $\mathcal{L}_{m}\cup\{ *\} $ obtained
from $m$ steps of the triangle removal process (where ``$*$''
corresponds to the event that we run out of triangles).
\end{definition}

\begin{definition}
Let $\mathcal{T}_{m}\subseteq\mathcal{L}_{m}$ be a property of $m$-edge
partial Latin squares and let $\mathcal{T}\subseteq\mathcal{L}$ be
a property of Latin squares. Say $\mathcal{T}_{m}$ is \emph{$\rho$-inherited}
from $\mathcal{T}$ if for any $L\in\mathcal{T}$, taking $\mathbf{L}_{m}\subseteq L$
as a uniformly random subset of $m$ edges of $L$, we have $\mathbf L_{m}\in\mathcal{T}_{m}$
with probability at least $\rho$.
\end{definition}

Now, our approximation theorem is as follows.

\begin{theorem}\label{thm:approximation}
Let $\alpha \in (0,1/2)$. There is an absolute constant $\gamma>0$ such that the following
holds. Consider $\mathcal{T}_{m}\subseteq\mathcal{L}_{m}$ with $m = \alpha n^2$ and $\mathcal{T}\subseteq\mathcal{L}$
such that $\mathcal{T}_{m}$ is $1/2$-inherited from $\mathcal{T}$.
Let $\mathbf{P}\sim\mathbb{L}(n,m)$ be a partial
Latin square obtained by $m$ steps of the triangle removal process,
and let $\mathbf{L}\in\mathcal{L}$ be a uniformly random order-$n$ Latin
square. Then
\[
\Pr(\mathbf{L}\in\mathcal{T})\le\exp(n^{2-\gamma})\Pr(\mathbf{P}\in\mathcal{T}_m).
\]
\end{theorem}

In our proof of \cref{thm:approximation}, we will need to refer to a number of general-purpose lemmas about random Latin squares and the triangle removal process, each of which essentially appears in \cite{Kwa20}. The lemmas in \cite{Kwa20} were stated in the setting of Steiner triple systems, but the necessary adaptations to the setting of Latin squares are straightforward. For completeness, in the time since the initial version of this paper we have prepared the companion note \cite{KSS21} with self-contained proofs of all the lemmas we will need, explicitly written for Latin squares.

\subsection{Counting completions of partial Latin squares} First, we need the fact that all partial Latin squares satisfying
a certain quasirandomness property extend to a Latin square in a comparable
number of ways. This is proved with the entropy method, the triangle removal process, and Keevash's machinery. First we define our notion of quasirandomness.
\begin{definition}
For this definition we write $V^{1},V^{2},V^{3}$
instead of $R,C,S$ for the three parts of $K_{n,n,n}$. A subgraph $G\subseteq K_{n,n,n}$ with $e(G)$ edges is \emph{$(\varepsilon,h)$-quasirandom}
if for each $i\in\{ 1,2,3\} $, every set $A\subseteq V\setminus V^i$
with $|A|\le h$ has $(1\pm\varepsilon)(e(G)/(3n^2))^{|A|}n$
common neighbours in $V^i$. For a partial Latin square $P\in\mathcal{L}_{m}$,
let $G(P)$ be the graph consisting of those edges of $K_{n,n,n}$
which are not included in some edge of $P$ (so if $m=n^{2}$ then
$G(P)$ is always the empty graph, and if $m=0$ then always $G(P)=K_{n,n,n}$).
Let $\mathcal{L}_{m}^{\varepsilon,h}$ be the set of partial Latin
squares $P\in\mathcal{L}_{m}$ such that $G(P)$ is $(\varepsilon,h)$-quasirandom.
\end{definition}

Second, it is convenient to define a notion of an \emph{ordered} (partial) Latin square.

\begin{definition}
An \emph{ordered} partial Latin square is a partial Latin square $P\in\mathcal{P}_{m}$
together with an ordering on its edge set. Since the triangle removal
process removes triangles sequentially, we can actually interpret
$\mathbb{L}(n,m)$ as a distribution on ordered partial
Latin squares with $m$ edges. Let $\mathcal{O}_{m}^{\varepsilon,h}$
be the set of ordered partial Latin squares $P\in\mathcal{L}_{m}$
such that, for each $i\le m$, writing $P_{i}$ for the partial Latin
square consisting of the first $i$ edges of $P$, the graph
$G(P_{i})$ is $(\varepsilon,h)$-quasirandom.
\end{definition}

Now, our counting lemma is as follows.

\begin{lemma}[{\cite[Lemma~1.6]{KSS21}}]\label{lem:count-completions}
For an ordered partial Latin square $P$ with $m$ edges, let $\mc O^*(P)$ be the set of ordered Latin squares extending $P$ (i.e., whose first $m$ edges are equal to $P$). Fixing a sufficiently large constant $h\in\mb N$ and fixing a constant $a>0$, there is $b=b(a,h)>0$
such that the following holds. Fix a constant $\alpha\in(0,1)$, let
$\varepsilon=n^{-a}$ and $m\le\alpha n^{2}$, and let $P,P'\in\mathcal{L}_{m}^{\varepsilon,h}$. Then
\[
\frac{|\mc O^*(P)|}{|\mc O^*(P')|}\le\exp(O(n^{2-b})).
\]
\end{lemma}


\subsection{The triangle removal process}

Next, we need the fact that the triangle removal process produces
every quasirandom partial Latin square with a comparable probability. This follows from the fact that quasirandom graphs have a predictable number of triangles.

\begin{lemma}[{\cite[Lemma~1.7]{KSS21}}]\label{lem:TRP-uniform}
The following holds for any fixed constant $a\in(0,2)$ and $\alpha\in(0,1)$. Let $\varepsilon=n^{-a}$,
let $P,P'\in\mathcal{O}_{\alpha m}^{\varepsilon,2}$ and let $\mathbf{P}\sim\mathbb{L}(n,\alpha m)$.
Then
\[
\frac{\Pr(\mathbf{P}=P)}{\Pr(\mathbf{P}=P')}\le\exp(O(n^{2-a})).
\]
\end{lemma}

We also need the fact that the triangle removal process is likely
to produce quasirandom partial Latin squares (and not output $\ast$). This follows from a
very simple and crude analysis (as in \cite[Section~6]{KSS21}). We note that
with modern techniques it is possible to prove a much stronger theorem
(see \cite{BFL15}), but this will not be necessary for our application.
\begin{lemma}[{\cite[Lemma~1.10]{KSS21}}]\label{lem:TRP-quasirandom}
For any constant $h\in\mb N$ there is a constant $a=a(h)\in (0,2)$ such that the following
holds. Fix $\alpha\in(0,1)$, let $m\le\alpha n^{2}$,
let $\varepsilon=n^{-a}$ and let $\mathbf{P}\sim\mathbb{L}(n,m)$.
Then $\Pr(\mathbf{P}\notin\mathcal{O}_{m}^{\varepsilon,h} \emph{ or } \mathbf{P} = \ast)=o(1)$.
\end{lemma}

\subsection{Randomly ordered Latin squares}

Finally, we need to know that a random ordering of any Latin square is likely to satisfy our quasirandomness property. This follows from a simple Chernoff bound calculation.
\begin{lemma}[{\cite[Lemma~1.8]{KSS21}}]\label{lem:random-ordering}
The following holds for any fixed constants $h\in\mathbb{N}$, $\alpha\in(0,1)$
and $a\in(0,1/2)$. Let $m\le\alpha n^{2}$ and $\varepsilon=n^{-a}$,
consider any Latin square $L$, and let $\mathbf{L}_{m}$ be a
random ordering of a random set of $m$ edges of $L$. Then $\Pr(\mathbf{L}_{m}\notin\mathcal{O}_{m}^{\varepsilon,h})=o(1)$.
\end{lemma}

\subsection{Putting everything together}Finally, we prove \cref{thm:approximation}.
\begin{proof}
[Proof of \cref{thm:approximation}] Let $h$ be as in \cref{lem:count-completions}, let $a=a(h)$ be as in \cref{lem:TRP-quasirandom}, let $\varepsilon=n^{-a}$, and let $b=b(a,h)$ be as in \cref{lem:count-completions}.

Let $N_{\mathrm{pair}}$ be the number of
pairs $(L,L_{m})$ where $L\in\mathcal{T}$ is a Latin square satisfying property $\mc T$, and $L_{m}$
is an ordered partial Latin square consisting of $m$ edges of $L$,
which satisfies\footnote{Here we are abusing notation slightly, because $\mc T_m$ is technically a property of \emph{unordered} partial Latin squares. Here we say an ordered partial Latin square satisfies $\mc T_m$ if its underlying unordered partial Latin square does.} $\mathcal{T}_{m}\cap\mathcal{O}_{m}^{\varepsilon,h}$.
Then 
\[
N_{\mathrm{pair}}\ge(1/2-o(1))n^{2}(n^{2}-1)\dots(n^{2}-m+1)|\mathcal{T}|,
\]
by \cref{lem:random-ordering} and the definition of being $1/2$-inherited.

Let $N_{\mathrm{ext}}=n^{2}(n^{2}-1)\dots(n^{2}-m+1)|\mathcal{L}|/|\mathcal{O}_{m}^{\varepsilon,h}|$
be an upper bound on the average number of ways to extend a partial
Latin square $P\in\mathcal{O}_{m}^{\varepsilon,h}$ to a Latin square.
By \cref{lem:count-completions}, we have
\[
N_{\mathrm{pair}}\le\exp(n^{2-b})|\mathcal{T}_{m}\cap\mathcal{O}_{m}^{\varepsilon,h}|N_{\mathrm{ext}}.
\]
It follows that
\[
\Pr(\mathbf{L}\in\mathcal{T})=\frac{|\mathcal{T}|}{|\mathcal{L}|}\le(2+o(1))\exp(n^{2-b})\frac{|\mathcal{T}_{m}\cap\mathcal{O}_{m}^{\varepsilon,h}|}{|\mathcal{O}_{m}^{\varepsilon,h}|}.
\]
Using \cref{lem:TRP-uniform}, we have
\[
\frac{|\mathcal{T}_{m}\cap\mathcal{O}_{m}^{\varepsilon,h}|}{|\mathcal{O}_{m}^{\varepsilon,h}|}\le\exp(O(n^{2-a}))\Pr(\mathbf{P}\in\mathcal{T}_{m}\,|\,\mathbf{P}\in\mathcal{O}_{m}^{\varepsilon,h}),
\]
and using \cref{lem:TRP-quasirandom}, we have
\[
\Pr(\mathbf{P}\in\mathcal{T}_{m}\,|\,\mathbf{P}\in\mathcal{O}_{m}^{\varepsilon,h})\le\frac{\Pr(\mathbf{P}\in\mathcal{T}_{m})}{\Pr(\mathbf{P}\in\mathcal{O}_{m}^{\varepsilon,h})}=(1+o(1))\Pr(\mathbf{P}\in\mathcal{T}_{m}).
\]
The desired result follows (taking $\gamma<\min\{b,a\}$).
\end{proof}

\section{Latin rectangles}

In this section we recall the notion of a Latin rectangle and some useful facts about them. The results in this section will be used in the proofs of the upper tail bounds \cref{thm:large-deviations}(b,d).

\begin{definition}\label{def:latin-rectangle-1}
A \emph{Latin rectangle} (of order $n$, with $k$ rows) is a $k\times n$ array containing the symbols $1,\dots,n$, such that every symbol appears at most once in each row and column. (So, if $k=n$ this is the same as a Latin square). A \emph{partial} Latin rectangle is a $k\times n$ array satisfying the same property, but where some of the cells are allowed to be empty.
\end{definition}

There is also an equivalent hypergraph formulation of a Latin rectangle.

\begin{definition}\label{def:latin-rectangle-2}
Recall the sets $R=R_n,C=C_n,S=S_n$ from \cref{def:latin-square} and, for $k\le n$, in addition define
\[
R^{(k)}=\{ 1,\dots,k\}\subseteq R.
\]
A \emph{partial Latin rectangle} (of order $n$, with $k$ rows) is
a 3-partite 3-graph with tripartition $R^{(k)}\cup C\cup S$
such that no pair of vertices is involved in more than one edge. A \emph{Latin rectangle} is a partial Latin rectangle with exactly $kn$ edges (which is the maximum possible). Let $\mc Q^{(k)}$ denote the set of all such Latin rectangles (we omit the superscript when $k$ is clear from context). We note that one can similarly define $C^{(k)}=\{n+1,\dots,n+k\}$ and $S^{(k)}=\{2n+1,\dots,2n+k\}$, and symmetrically define a notion of a Latin rectangle with $k$ columns or with $k$ symbols.
\end{definition}

We will switch back and forth between the two equivalent definitions in \cref{def:latin-rectangle-1,def:latin-rectangle-2}, depending on which is more convenient at the time (this will be clear from context).

\subsection{Counting completions of Latin rectangles}
The primary reason Latin rectangles will be important for us is that every Latin rectangle can be completed in roughly the same number of ways to a Latin square. The following lemma is from \cite[Proposition~4]{MW99}. It is proved by iteratively applying Bregman's theorem~\cite{Bre73} and the Egorychev--Falikman theorem~\cite{Ego81,Fal81} to give upper and lower bounds on the number of ways to add an extra row to a given Latin rectangle.
\begin{theorem}\label{thm:permanent}
Let $Q,Q'$ be two Latin rectangles with order $n$ and the same number $k$ of rows. Let $\mathbf L$ be a random Latin square and let $\mathbf L_k$ be the Latin rectangle consisting of its first $k$ rows. Then
\[\frac{\Pr(\mathbf L_k=Q)}{\Pr(\mathbf L_k=Q')}=e^{O(n(\log n)^2)}.\]
\end{theorem}
\subsection{Subset probabilities in random Latin rectangles}
The following theorem provides estimates on the probability that a given set of entries is present in a random Latin rectangle. It is a direct consequence of a theorem of Godsil and McKay (\cite[Theorem 4.7]{GM90}), and is proved using the switching method.
\begin{theorem}\label{thm:GM}
Let $P$ be a partial Latin rectangle, let $d_{i}(P)$ denote the number of entries of $P$ in row $i$, and let
$\Delta=\max_{1\le i\le k}d_{i}(P)$. Let $\mathbf{L}$
be a uniformly random $k\times n$ Latin rectangle and suppose $\Delta\le n-5k$. Then
\[
\Pr(P\subseteq\mathbf{L})=\left(\frac{1+O(k/(n-2k-\Delta))}{n}\right)^{|P|}.
\]
\end{theorem}

\section{Lower bounds}
In this section we prove \cref{thm:large-deviations}(c--d), lower-bounding the large deviation probabilities for the number of intercalates $\mathbf N$ in a random Latin square.
\begin{proof}[Proof of \cref{thm:large-deviations}(c)]
As noted in the introduction, for all orders except $2$ and $4$ there is a Latin square with no intercalates (this combines results of Kotzig, Lindner and Rosa~\cite{KLR75}, McLeish~\cite{McL75}, and Kotzig and Turgeon~\cite{KT76}). On the other hand, the total number of order-$n$ Latin squares is clearly at most $n^{n^2}=\exp(O(n^2\log n))$. The desired result follows.
\end{proof}

\begin{proof}[Proof of \cref{thm:large-deviations}(d)]
Let $k=2^q$ be the smallest power of two such that $k\binom k 2/2\ge (1+\delta)n^2/4$. Let $L$ be the Latin square corresponding to the multiplication table of $(\mb Z/2\mb Z)^q$ (where we fix some correspondence between elements of this group and the integers $1,\dots,k$). Then, $L$ has order $k=\Theta(n^{2/3})$, and it is easy to see that it has $k\binom k 2/2\ge (1+\delta)n^2/4$ intercalates (see for example \cite{BCW14}). Let $\mathbf Q\in \mc Q$ be a uniformly random order-$n$ Latin rectangle with $k$ rows. By \cref{thm:GM}, with probability at least $((1-o(1))/n)^{k^2}=\exp(-O(n^{4/3}\log n))$, our special Latin square $L$ appears in the first $k$ columns of $\mathbf Q$.

Let $\mathbf L_r$ be the Latin rectangle consisting of the first $k$ rows of our random Latin square $\mathbf L$. By \cref{thm:permanent}, the probabilities of different outcomes of $\mathbf L_r$ differ by a factor of only $e^{O(n(\log n)^2)}$, so with probability at least $\exp(-O(n^{4/3}\log n+n(\log n)^2))=\exp(-O(n^{4/3}\log n))$, our special Latin square $L$ appears in the first $k$ rows and columns of $\mathbf L$, in which case $\mathbf N\ge (1+\delta)n^2/4$.
\end{proof}

\section{Upper-bounding the lower tail}

In this section we prove \cref{thm:large-deviations}(a). We will apply \cref{thm:approximation} with $\mc T$ being the property of having ``too few'' intercalates. First, we establish that this property is likely to be inherited by random subsets.

\begin{lemma}\label{lem:inherited}
Fix $\alpha,\delta\in [0,1]$, let $\mc T^\delta\subseteq \mc L$ be the property that a Latin square $L\in \mc L$ has at most $(1-\delta)n^2/4$ intercalates, and for $m=\alpha n^2$ let $\mc T^\delta_m\subseteq \mc L_m$ be the property that a partial Latin square $P\in \mc L_m$ has at most $\alpha^4(1-\delta/2)n^2/4$ intercalates. Then $\mc T^\delta_m$ is $1/2$-inherited from $\mc T^\delta$.
\end{lemma}
\begin{proof}
Let $L\in T^\delta$ and let $\mathbf L_m$ consist of $m$ random edges of $L$. Let $\mc I$ be the set of intercalates in $L$, for $I\in \mc I$ let $\mathbf 1_I$ be the indicator random variable for the event that $I\subseteq L_m$, and let $\mathbf X=\sum_{I\in\mc I}\mathbf 1_I$ be the number of intercalates in $\mbf L_m$.
For each $I\in \mc I$ we have $\mb E \mathbf 1_I=\alpha^4+O(1/n)$, so $\mb E \mathbf X\le \alpha^4(1-\delta+o(1))n^2/4$. Also, for each pair of disjoint $I,J\in \mc I$ we have $\on{Cov}(\mathbf 1_I,\mathbf 1_J)=O(1/n)$.
In every Latin square, every intercalate intersects at most $4n$ other intercalates, so there are $O(n^3)$ intersecting pairs of intercalates in $\mc I$, meaning $\on{Var}\mathbf X=O(n^3)$. By Chebyshev's inequality, we conclude that
\[\Pr(\mathbf L_m\in \mc T)=\Pr(\mathbf X<\alpha^4(1-\delta/2)n^2/4)=1-o(1)>1/2,\]
meaning that $\mc T^\delta_m$ is $1/2$-inherited from $\mc T^\delta$.
\end{proof}

Before we continue with the proof we record some auxiliary lemmas.

\subsection{A coupling lemma}
It is not very easy to study the triangle removal process directly, so the following coupling lemma is useful in combination with \cref{thm:approximation}. Let $\mb G^{(3)}(n,p)$
be the random 3-partite 3-graph on the vertex set $R\cup C\cup S$
obtained by including all possible edges with probability $p$
independently.
\begin{lemma}[{\cite[Lemma~1.9]{KSS21}\footnote{The statement of \cite[Lemma~1.9]{KSS21} is for a monotone \emph{increasing} property $\mc P$; to derive the statement here we simply take $\mc P$ to be the complement of $\mc T$.}}]\label{lem:coupling}
Let $\mathcal{T}$ be a property of unordered partial Latin squares
that is monotone decreasing in the sense that $P\in\mathcal{T}$ and
$P'\subseteq P$ implies $P'\in\mathcal{T}$. Fix $\alpha\in(0,1)$, let $\mathbf{P}\sim\mb L(n,\alpha n^{2})$, let $\mathbf{G}\sim\mb G^{(3)}(n,\alpha/n)$
and let $\mathbf{G}^{*}$ be the partial Latin square obtained from $\mathbf{G}$ by
deleting (all at once) every edge which intersects another edge in more than one
vertex. Then 
\[
\Pr(\mathbf{P}\in\mathcal{T})\le O(\Pr(\mathbf{G}^{*}\in\mathcal{T})).
\]
\end{lemma}
We remark that one can prove a similar coupling lemma for monotone increasing properties (see \cite[Lemma~2.6]{FK20}), though this will not be necessary for us.

\subsection{A concentration inequality}
The following concentration inequality may be deduced from an inequality of Freedman~\cite{Fre75}. It appears as \cite[Theorem~2.11]{Kwa20}.
\begin{theorem}\label{thm:concentration}
Let $\boldsymbol\omega=(\boldsymbol \omega_1,\dots,\boldsymbol\omega_N)$ be a sequence of independent, identically distributed random variables with $\Pr(\boldsymbol\omega_i=1)=p$ and $\Pr(\boldsymbol\omega_i=0)=1-p$. Let $f:\{0,1\}^N\to \mb R$ satisfy the Lipschitz condition $|f(\boldsymbol \omega)-f(\boldsymbol\omega')|\le K$ for all pairs $\bs \omega,\bs \omega'\in \{0,1\}^N$ differing in exactly one coordinate. Then
\[\Pr(|f(\bs \omega)-\mb E f(\bs \omega)|>t)\le \exp\left(-\frac{t^2}{4K^2Np+2Kt}\right).\]
\end{theorem}
\subsection{Putting everything together}We are now ready to prove \cref{thm:large-deviations}(a).
\begin{proof}[Proof of \cref{thm:large-deviations}(a)]
Let $\alpha>0$ be some constant that is sufficiently small with respect to $\delta$, let $m=\alpha n^2$, let $\mathbf G,\mathbf G^*$ be as in \cref{lem:coupling}, and let $\mc T^\delta_m$ be the property that a partial Latin square (not necessarily with exactly $m$ edges) has at most $(1-\delta/2)\alpha^4n^2/4$ intercalates. We bound $\Pr(\mbf G^* \in \mc T^\delta_m)$ using a ``maximum disjoint family'' technique essentially due to Bollob\'as~\cite{Bol88}. Let $\mathbf N$ be the number of intercalates in $\mathbf G^*$, let $\mathbf N'$ be the maximum size of a collection of disjoint intercalates in $\mathbf G^*$, and let $\mathbf N_2$ be the number of pairs of distinct intercalates in $\mathbf G$ which share an edge. Observe that $\mathbf N'\ge \mathbf N-\mathbf N_2$.

Now, we estimate  $\mb E\mathbf N$. There are $2\binom n 2^3$ ways to specify an intercalate, and each is present in $\mathbf G^*$ with probability $(\alpha/n)^4(1-\alpha/n)^{12(n-1)-8}$.
So,
\[\mb E \mathbf N=(e^{-12\alpha}+o(1))\alpha^4 n^2/4.\]
For $\mb E \mathbf N_2$, we observe that there are $O(n^7)$ ways to specify a pair of intercalates that share two edges, and each such pair is present in $\mathbf G$ with probability $(\alpha/n)^6$. There are $2\binom{n}{2}^3\cdot 4(n-2)^3$ ways to specify an ordered pair of intercalates that share one edge, and each such pair is present in $\mathbf G$ with probability $(\alpha/n)^7$.
So,
\[\mb E\mathbf N_2=(\alpha^3+o(1))\alpha^4 n^2.\]
If $\alpha$ is sufficiently small (in terms of $\delta$) then $\mb E \mathbf N'\ge \mb E \mathbf N-\mb E\mathbf N_2\ge \alpha^4(1-\delta/4)n^2/4$.

We next claim that $\mathbf N'$ is a 3-Lipschitz function of the edges of the random hypergraph $\mathbf G$. Indeed, adding an edge to $\mathbf G$ can increase $\mathbf N'$ by at most one, and removing an edge from $\mathbf G$ can increase $\mathbf N'$ by at most three (by adding up to three edges to $\mathbf G^*$). So, by \cref{thm:concentration} we have
\[\Pr(\mathbf G^*\in \mc T^\delta_m)\le \Pr(\mathbf N'\le \mb E\mathbf N'-\alpha^4\delta n^2/16)\le \exp(-\Omega(n^2)).\]
It follows from \cref{lem:coupling} that if $\mathbf P\sim\mb{L}(n,\alpha n^2)$ then $\Pr(\mathbf P\in \mc T^\delta_m)\le\exp(-\Omega(n^2))$. The desired result follows from \cref{lem:inherited} and \cref{thm:approximation}.
\end{proof}

\section{Upper-bounding the upper tail}
In this section we prove \cref{thm:large-deviations}(b). We will work mostly with random Latin rectangles, and use \cref{thm:permanent} to transfer our results to random Latin squares. Recall that we defined two equivalent notions of a Latin rectangle (\cref{def:latin-rectangle-1,def:latin-rectangle-2}); we will use both perspectives in this section.

\subsection{Deletion}
The first step in the proof of \cref{thm:large-deviations}(b) is to adapt the deletion method of R\"odl and Ruci\'nski (see \cite{RR95,JR04}), using \cref{thm:GM}, to reduce to the case where one has a small subset of edges which contributes a large number of the intercalates. To effectively apply \cref{thm:GM}, for now we restrict our attention to a small number of rows, columns and symbols.

\begin{lemma}\label{lem:deletion}
Fix a sufficiently small constant $\delta>0$, and let $k=\lfloor\delta^2 n\rfloor$. Let $\mbf L$ be a uniformly random order-$n$ Latin square, and let $\mbf L^{(k)}$ be the subhypergraph induced by $R^{(k)}\cup C^{(k)}\cup S^{(k)}$ (that is to say, $\mbf L^{(k)}$ consists of the entries in the first $k$ columns and the first $k$ rows, involving the first $k$ symbols). Then with probability $1-\exp(-\Omega(n^{4/3}(\log n)^{2/3}))$ there is a set $\mbf E_0\subseteq \mbf L^{(k)}$ of $n^{4/3}(\log n)^{2/3}$ edges such that $\mbf L^{(k)}\setminus \mbf E_0$ contains at most $(1+\delta/2)k^6/(4n^4)$ intercalates.
\end{lemma}

\begin{proof}
We first note that this event is purely a function of the first $k$ rows of $\mbf L$, and therefore by \cref{thm:permanent} it suffices to prove the same event for a uniformly random order-$n$ Latin rectangle $\mbf Q\in \mc Q$ with $k$ rows (as the relative change of measure $\exp(O(n(\log n)^2))$ is simply swallowed into the error term). Let $\mbf Q^{(k)}$ be the subhypergraph of $\mbf Q$ induced by $R^{(k)}\cup C^{(k)}\cup S^{(k)}$.

Now let $\mathcal A$ be the event that the desired property fails (i.e., for each set $E\subseteq \mbf Q^{(k)}$ of size $r=n^{4/3}(\log n)^{2/3}$, the partial Latin rectangle $\mbf Q^{(k)}\setminus E$ contains at least $N=(1+\delta/2)k^6/(4n^4)$ intercalates). Let $Z$ be the number of $\kappa=\lfloor r/4\rfloor$-element sequences of disjoint intercalates in $\mbf Q^{(k)}$. If $\mathcal A$ holds, then $Z\ge N^\kappa$, since we may choose $\kappa$ intercalates sequentially with at least $N$ choices each time.
On the other hand, there are $2\binom k 2^3=(1+o(1))k^6/4$ potential intercalates that can appear in $\mbf Q^{(k)}$, and by \cref{thm:GM} each $\kappa$-element sequence of disjoint intercalates appears in $\mbf Q$ with probability at most $((1+O(\delta^2))/n^4)^{\kappa}\le((1+\delta/4)/n^4)^{\kappa}$ (for small $\delta$). So, $\mb{E} Z\le \left((1+\delta/3)k^6/(4n^4)\right)^\kappa$ by linearity of expectation, and by Markov's inequality it follows that $\Pr(\mathcal A)\le \mb E Z/N^{\kappa}\le \exp(-\Omega(n^{4/3}(\log n)^{2/3}))$.
\end{proof}

\subsection{A combinatorial decomposition}
Given \cref{lem:deletion}, we now wish to understand the probability that there is a small set of edges participating in many intercalates. To make this analysis tractable, we need a lemma decomposing any set of edges into well-behaved subsets.
\begin{definition}
A \emph{star} is a hypergraph all of whose hyperedges contain a common vertex.
A \emph{matching} is a hypergraph all of whose hyperedges are disjoint.
\end{definition}

\begin{lemma}\label{lem:star-matching}
For any $r\in \mb N$, every $3$-uniform hypergraph with $m$ hyperedges can be partitioned into a combination of at most $m/r$ stars and at most $3r+m/r$ matchings, each of which have at most $r$ edges.
\end{lemma}
\begin{proof}
As long as there is a vertex incident to $r$ edges, take $r$ of these edges as a star (we obtain at most $m/r$ stars in this way). After no more deletions are possible, we now have a $3$-uniform hypergraph with all degrees less than $r$. We can greedily find a proper edge-colouring of this hypergraph with at most $3(r-1)+1\le 3r$ colours. Each of the colour classes is a matching. Finally, arbitrarily decompose the matchings into sub-matchings each with at most $r$ edges, which introduces at most $m/r$ new matchings.
\end{proof}
\subsection{Switching for stars and matchings}
By applying \cref{lem:star-matching} to the set $\mbf E_0$ provided by \cref{lem:deletion}, it now suffices to bound the probability that there is a small star or matching which participates in many intercalates. We will handle both cases separately, with similar switching-based proofs (in random Latin rectangles; afterwards we will use \cref{thm:permanent} to deduce a result for random Latin squares). Our application of the switching method will be rather simple and completely elementary, but we remark that Fack and McKay~\cite{FM07} and Hasheminezhad and McKay~\cite{HM10} have proved very general theorems with which one can analyse more complicated switching operations.

First, the following lemma will be used to handle stars (note that in the context of Latin rectangles, a star is a set of entries corresponding to a single row, column or symbol).
\begin{lemma}\label{lem:star-switching}
Let  $k\le n/10$, and let $\mbf{Q}\in \mc Q$ be a uniformly random order-$n$ Latin rectangle with $k$ rows. Let $N(\mbf Q)$ be the number of intercalates in $\mbf Q$ which involve the first row. Then $\Pr(N(\mbf Q)\ge t)\le (k/t)^{\Omega(t)}$ for $t\ge 20k$.
\end{lemma}
\begin{proof}
Let $\mc Q(\ell)\subseteq \mc Q$ be the set of Latin rectangles $Q\in \mc Q$ for which there are exactly $N(Q)=\ell$ intercalates involving the first row.

Consider the following switching operation: select a row $i\in R^{(k)}\setminus\{1\}$ (i.e., not the first row) and a pair of columns $x,y\in C$, and swap the contents of columns $x$ and $y$ in row $i$. Note that it is possible that the resulting $k\times n$ array is no longer a Latin rectangle (columns $x$ and $y$ may now contain a repeated symbol). We next compute some upper and lower bounds on the number of ways to switch from a Latin rectangle $Q\in \mc Q(\ell)$ to a Latin rectangle $Q\in \mc Q(\ell')$, for $\ell\ne \ell'$.

In the hypergraph formulation of a Latin rectangle, our switching introduces two new edges and removes two edges. In a Latin rectangle, every edge outside the first row participates in at most one intercalate involving the first row, so it is only possible to switch from $\mc Q(\ell)$ to $\mc Q(\ell')$ if $|\ell-\ell'|\le 2$.

Next, we observe that for any Latin rectangle $Q\in \mc Q$, there are at most $(k-1)n\le kn$ switchings which create an intercalate involving the first row. Indeed, first note that swapping entries in columns $x$ and $y$ of a given row can never create an intercalate involving the first row and both $x$ and $y$. Then, for every column $z$, we consider the number of switchings that create an intercalate involving $z$ without actually swapping an entry in column $z$. Such an intercalate must involve one of the $k-1$ rows other than the first, and for each such row $i$, there is at most one switching that actually creates the desired intercalate (in row $i$, we must swap the column $x$ satisfying $Q_{i,x}=Q_{1,z}$ with the column $y$ satisfying $Q_{1,y}=Q_{i,z}$).

Now, given a Latin rectangle $Q\in \mc Q(\ell)$, there are $\ell$ intercalates involving the first row. Note that each such intercalate is destroyed by at least $2(n-2k)$ switchings which maintain the Latin rectangle property. Indeed, consider one of the two edges of the intercalate not in the first row (in row $i$, column $x$ and symbol $s$, say). There are at most $k$ columns which already include $s$, and at most $k$ columns whose symbol in row $i$ already appears in column $x$. For any of the (at least $n-2k$) other rows $y$, we can swap columns $x$ and $y$ in row $i$ to destroy the desired intercalate. Now, the intercalates involving the first row are edge-disjoint outside of this first row, but a given switching could remove two different intercalates at once (or destroy an intercalate by interchanging its entries outside the first row). So, there are a total of at least $\ell(n-2k)$ distinct switchings which maintain the Latin rectangle property and remove an intercalate. We have just observed that at most $kn$ of these switchings also introduce an intercalate, so for any $\ell\ge 2$ we deduce
\[\big(\ell(n-2k)-kn\big)\,|\mc Q(\ell)|\le kn\,(|\mc Q(\ell-1)|+|\mc Q(\ell-2)|).\]
This implies $|\mc Q(\ell)|\le(4k/\ell)\max\{|\mc Q(\ell-1)|,|\mc Q(\ell-2)|\}$ for $\ell\ge 10k$. Iterating this, we see that for $\ell\ge 20k$ we have
\begin{align*}|\mc Q(\ell)|&\le \frac{4k}\ell\cdot\frac{4k}{\ell-2}\cdot\dots\cdot\frac{4k}{\ell-2\lceil(\ell-10k-2)/2\rceil}\cdot \max\{|\mc Q(10k)|,|\mc Q(10k-1)|\}\le\bigg(\frac{k}{\ell}\bigg)^{\Omega(\ell)}|\mc Q|.
\end{align*}
(To justify the second inequality, note that at least $\ell/4$ terms in the product are at most $8k/\ell$).
This implies that $\Pr(N(\mbf Q)=\ell)\le (k/\ell)^{\Omega(\ell)}$, and the desired result follows by summing over $\ell\ge t$.
\end{proof}

To handle matchings, we first use a similar switching argument to handle intercalates which are ``mostly disjoint'' from the vertices of the matching, other than the necessary included edge. Note that in the context of Latin rectangles, a matching is a set of entries such that no pair shares a row, column or symbol. Such a set is also called a \emph{partial transversal}.
\begin{lemma}\label{lem:initial-matching-switching}
Let $k\le n/10$ and let $\mbf{Q}\in \mc Q$ be a uniformly random order-$n$ Latin rectangle with $k$ rows. Fix a set $M$ of $r=k/6$ disjoint triples in $R^{(k)}\times C\times S$ (which may or may not appear as edges in $\mbf Q$). Say an intercalate is \emph{good} if it includes one of the triples in $M$ as an edge, and its other three vertices are completely disjoint from the vertices in $M$. Let $N(\mbf Q)$ be the number of good intercalates in $\mbf Q$. Then $\Pr(N(\mbf Q)\ge t)\le (k/t)^{\Omega(t)}$ for $t\ge 20k$.
\end{lemma}
\begin{proof}
Similarly to the proof of \cref{lem:star-switching}, we partition the set of all $k\times n$ Latin rectangles $\mc Q$ into subsets $\mc Q(\ell)$ depending on the number $\ell$ of good intercalates they contain. Without loss of generality we may assume that $M$ involves the first $r$ rows, the first $r$ columns and the first $r$ symbols. We consider the same switching operation as before, but we only consider swaps in rows $i>r$ (i.e., we never switch in the rows where the entries of $M$ live). As in the proof of \cref{lem:star-switching}, we need to prove estimates on the number of ways to switch between different $\mc Q(\ell)$. The arguments will be very similar, so we will be brief with the details.

This time, it is only possible to switch between $\mc Q(\ell)$ and $\mc Q(\ell')$ if $|\ell-\ell'|\le 6$. This is because any given entry outside the first $r$ rows can be involved in at most $3$ good intercalates (it must share a row, column or symbol with an edge in $M$).

The same considerations as before show that for any $Q\in \mc Q$, there are at most $rk$ switchings which create a good intercalate. Also, if we consider any $Q\in \mc Q(\ell)$, there are at least $\ell(n-2k)/3$ switchings which destroy an intercalate in $Q$ (the reason we divide by three is that a single entry can participate in at most three good intercalates). We deduce that, for $\ell\ge 6$,
\[\big(\ell(n-2k)/3-rk\big)\,|\mc Q(\ell)|\le rk\sum_{r=1}^6|\mc Q(\ell-r)|,\]
and we can then iterate this bound to conclude the proof in essentially the same way as \cref{lem:star-switching}.
\end{proof}
Now, a simple combinatorial argument allows us to infer a bound not requiring disjointness.
\begin{lemma}\label{lem:matching-total}
There is a constant $C_{\ref{lem:matching-total}}>0$ such that the following holds. Let $k\le n/10$ and let $\mbf{Q}\in \mc Q$ be a uniformly random order-$n$ Latin rectangle with $k$ rows. Fix a set $F$ of $r\le k/6$ disjoint triples in $R^{(k)}\times C\times S$ (which may or may not appear as edges in $\mbf Q$). Let $N(\mbf Q)$ be the number of intercalates in $\mbf Q$ which include an edge in $F$. Then $\Pr(N(\mbf Q)\ge s)\le (k/s)^{\Omega(s)}$ for $s\ge C_{\ref{lem:matching-total}}k$.
\end{lemma}
\begin{proof}
We claim that if any Latin rectangle $Q\in \mc Q$ has at least $s\ge C_{\ref{lem:matching-total}}k$ intercalates involving edges in $F$, then there is a subset $M\subseteq Q$ such that there are at least $s/16$ intercalates in $Q$ which are good with respect to $M$ (i.e., they involve an edge in $M$, and the three vertices outside this edge are completely disjoint from $M$).
This suffices to prove the lemma: if $H$ is large enough then \cref{lem:initial-matching-switching} and the union bound show that with probability at least $1-2^r(16k/s)^{\Omega(s/16)}=1-(k/s)^{\Omega(s)}$, our random Latin rectangle $\mbf R$ has the property that there is no subset $M\subseteq Q$ for which there are at least $s/16$ intercalates in $\mbf R$ which are good with respect to $M$.

To prove the claim, we use the probabilistic method. Consider any Latin rectangle $Q\in \mc Q$, and let $M$ be a random subset of $F$ obtained by including each element of $F$ independently with probability $1/2$. For each intercalate $I$ involving an edge in $e\in F$, note that $I$ is good with respect to $M$ with probability at least $1/16$. Indeed, note that there are at most $3$ edges in $F\setminus \{e\}$ which intersect $I$. The probability that $e\in M$ and the other intersecting edges are not in $M$ is at least $(1/2)\cdot (1/2)^3=1/16$.

By linearity of expectation, the expected number of intercalates which are good with respect to $M$ is at least $s/16$, so there is an outcome of $M$ such that there are at least $s/16$ good intercalates. This completes the proof of the claim.
\end{proof}

We conclude this subsection by using \cref{thm:permanent} to deduce from \cref{lem:star-switching,lem:matching-total} a corresponding result for random Latin squares.

\begin{lemma}\label{lem:star-matching-combined}
Let $k\le n/10$ and $r\le k/6$. Let $\mbf{L}\in \mc L$ be a uniformly random order-$n$ Latin square, and let $\mbf L^{(k)}$ be the subhypergraph induced by $R^{(k)}\cup C^{(k)}\cup S^{(k)}$ (i.e., the first $k$ rows, columns and symbols). Let $K^{(3)}_{k,k,k}$ be the complete 3-uniform 3-partite hypergraph with parts $R^{(k)},C^{(k)},S^{(k)}$, and fix a star or matching $F\subseteq K^{(3)}_{k,k,k}$ with $r$ edges. Let $\mbf N_F$ be the number of intercalates in $\mbf L^{(k)}$ which include an edge in $F$. Then $\Pr(\mbf N_F\ge s)\le \exp(O(n(\log n)^2))(k/s)^{\Omega(s)}$ for $s\ge C_{\ref{lem:matching-total}}k$, where $C_{\ref{lem:matching-total}}$ is the constant in \cref{lem:matching-total}.
\end{lemma}
\begin{proof}
If $F$ is a matching, we may assume without loss of generality that it involves the first $r$ rows. The desired result then follows from \cref{lem:matching-total} and \cref{thm:permanent} (recall that by \cref{thm:permanent}, we lose a factor of at most $\exp(O(n(\log n)^2))$ when changing measure from a random $k\times n$ Latin rectangle to the first $k$ rows of a random Latin square).

If $F$ is a star, without loss of generality we may assume that all of its edges are in the first row (recall that there is a symmetry between the rows, columns and symbols of a Latin square). We then apply \cref{lem:star-switching} and \cref{thm:permanent} in the same way.
\end{proof}

\subsection{Completing the proof}
We are now ready to bound the upper tail deviation probability.
\begin{proof}[Proof of \cref{thm:large-deviations}(b)]
We may assume $\delta$ is sufficiently small (the desired bound only becomes stronger as we make $\delta$ smaller). Let $k=\lfloor\delta^2 n\rfloor$, and let $\mbf L^{(k)}$ be the subhypergraph induced by $R^{(k)}\cup C^{(k)}\cup S^{(k)}$ (i.e., the first $k$ rows, columns and symbols). Let $\mbf N_k$ be the number of intercalates in $\mbf L^{(k)}$; we will prove that $\Pr(\mbf N_k\ge(1+2\delta/3)k^6/(4n^4))\le \exp(-\Omega(n^{4/3} (\log n)^{2/3}))$. To see that this suffices, note that by symmetry and the union bound, it would follow that $\binom nk^3\exp(-\Omega(n^{4/3} (\log n)^{2/3}))=\exp(-\Omega(n^{4/3} (\log n)^{2/3}))$ is an upper bound on the probability that there is \emph{any} choice of $k$ rows, columns and symbols which contains more than $(1+2\delta/3)k^6/(4n^4)$ intercalates. But if an order-$n$ Latin square $L$ contains at least $(1+\delta)n^2/4$ intercalates, then by averaging there is some subset of $k$ rows, $k$ columns and $k$ symbols inducing at least $(1+2\delta/3)k^6/(4n^4)$ intercalates.

So, we study intercalates in $\mbf L^{(k)}$. Let $K^{(3)}_{k,k,k}$ be the complete 3-uniform 3-partite hypergraph with parts $R^{(k)},C^{(k)},S^{(k)}$, and for a set of edges $E\subseteq K^{(3)}_{k,k,k}$, let $\mbf N_E$ be the number of intercalates in $\mbf L^{(k)}$ involving an edge of $E$. For every possible outcome of $\mbf L^{(k)}$, let $\mbf E_0\subseteq \mbf L^{(k)}$ be a subset of $m=n^{4/3}(\log n)^2/3$ edges of $\mbf L^{(k)}$ such that $\mbf N_{\mbf E_0}$ is maximised. By \cref{lem:deletion}, it suffices to show that $\Pr(\mbf N_{\mbf E_0}\ge (\delta/6)k^6/(4n^4))\le \exp(-\Omega(n^{4/3} (\log n)^{2/3}))$.

Let $r=\sqrt m=n^{1/3}(\log n)^{1/3}$. By \cref{lem:star-matching}, we can always partition $\mbf E_0$ into at most $3r+m/r=4r$ stars and matchings each with at most $r$ edges. In order to have $\mbf N_{\mbf E_0}\ge (\delta/6)k^6/(4n^4)$, there must be some $r$-edge star or matching $F$ with
\[\mbf N_F\ge \frac{(\delta/6)k^6/(4n^4)}{4r}=\Omega\big(n^{4/3}(\log n)^{-1/3}\big).\]
But this occurs with probability at most $\binom{n^3}{k}\exp(-\Omega(n^{4/3} (\log n)^{2/3}))=\exp(-\Omega(n^{4/3} (\log n)^{2/3}))$ by \cref{lem:star-matching-combined} and the union bound.
\end{proof}

\bibliographystyle{amsplain_initials_nobysame_nomr}
\bibliography{main.bib}

\end{document}